\documentclass{article}
\usepackage{amsmath,amssymb,amsthm,graphicx,epsfig,float,url}
\usepackage[colorlinks=true,citecolor={Plum},linkcolor={Periwinkle}]{hyperref}
\usepackage{pdfsync}
\usepackage[usenames, dvipsnames]{xcolor}
\usepackage{tikz}
\usepackage{subfig}
\usepackage{bbm}
\usepackage{stmaryrd}
\usepackage{amssymb}
\usepackage{mathrsfs}  
\usepackage{caption}
\usepackage{float}
\usepackage{esint}

\usepackage{dsfont}
\usepackage[makeroom]{cancel}
\usetikzlibrary{patterns}
\newcommand{\R}{\textnormal{I\kern-0.21emR}}
\newcommand{\N}{\textnormal{I\kern-0.21emN}}

\renewcommand{\geq}{\geqslant}
\renewcommand{\leq}{\leqslant}

\def\B{{\mathbb B}}

\allowdisplaybreaks

\def\YYint#1#2#3{{\setbox0=\hbox{$#1{#2#3}{\iint}$}
    \vcenter{\hbox{$#2#3$}}\kern-.51\wd0}}
 
\usepackage[dvipsnames]{xcolor}

\newtheorem*{theorem*}{Theorem}

\newtheorem{theorem}{Theorem}

\newtheorem{material}{material}

\newtheorem{definition}[material]{Definition}
\newtheorem{lemma}[material]{Lemma}

\def\O{{\Omega}}

\def\n{{\nabla}}

\def\p{{\varphi}}

\def\OT{{(0;T)\times \O}}

\usepackage{xargs}
 \usepackage[colorinlistoftodos,textsize=small]{todonotes}
 \newcommandx{\christian}[2][1=]{\todo[linecolor=red,backgroundcolor=red!25,bordercolor=red,#1]{#2}}
 \newcommandx{\laura}[2][1=]{\todo[linecolor=blue,backgroundcolor=blue!25,bordercolor=blue,#1]{#2}}
 \newcommandx{\info}[2][1=]{\todo[linecolor=green,backgroundcolor=green!25,bordercolor=green,#1]{#2}}
 \newcommandx{\improvement}[2][1=]{\todo[linecolor=yellow,backgroundcolor=yellow!25,bordercolor=yellow,#1]{#2}}
 
  \newcommandx{\biblio}[2][1=]{\todo[linecolor=blue,backgroundcolor=magenta!25,bordercolor=blue,#1]{#2}}

 \numberwithin{equation}{section}

\begin{document}

\title{A note on the rearrangement of functions in time and on the parabolic Talenti inequality}


\author{Idriss Mazari\footnote{CEREMADE, UMR CNRS 7534, Universit\'e Paris-Dauphine, Universit\'e PSL, Place du Mar\'echal De Lattre De Tassigny, 75775 Paris cedex 16, France, (mazari@ceremade.dauphine.fr).}}
\date{ \today}

\maketitle
\begin{abstract}
Talenti inequalities are a central feature in the qualitative analysis of PDE constrained optimal control as well as in calculus of variations. The classical parabolic Talenti inequality states that if we consider the parabolic equation ${\frac{\partial u}{\partial t}}-\Delta u=f=f(t,x)$ then, replacing, for any time $t$, $f(t,\cdot)$ with its Schwarz rearrangement $f^\#(t,\cdot)$ increases the concentration of the solution in the following sense: letting $v$ be the solution of ${\frac{\partial v}{\partial t}}-\Delta v=f^\#$ in the ball, then the solution $u$ is less concentrated than $v$. This property can be rephrased in terms of the existence of a maximal element for a certain order relationship. It is natural to try and rearrange the source term not only in space but also in time, and thus to investigate the existence of such a maximal element when we rearrange the function with respect to the two variables. In the present paper we prove that this is not possible.
\end{abstract}


\medskip


\medskip

\noindent
\textbf{Acknowledgment:} The author was partially supported by the French ANR Project ANR-18-CE40-0013 - SHAPO on Shape Optimization and by the Project ”Analysis and simulation of optimal shapes - application to life sciences” of the Paris City Hall. The author would like to thank the anonymous referee for his or her numerous suggestions which helped improve the manuscript.

\section{Introduction and motivation}
\subsection{Scope of the paper and mathematical context}
In this paper we want to address some qualitative questions related to the time-rearrangement of functions in the context of optimal control and Talenti inequalities. Roughly speaking, it has been known since the seminal {paper } \cite{Talenti} that the spatial rearrangement (\emph{i.e.} the Schwarz rearrangement) of source terms in elliptic equations improved "concentration"-like properties. Before we make this statement more precise let us note that this work of Talenti has sparked an immense interest from the calculus of variations and optimisation community, leading to major developments, whether in calculus of variations, in optimal control or in fine comparison relations for parabolic and elliptic partial differential equations \cite{Alvino1986,Alvino1990,alvino1991,Alvino2010,AlvinoNitschTrombetti,AlvinoTrombettiLions,Bandle,Chiacchio,Hamel2011,Kesavan,Kesavan1988,Langford,Mazari2022,Mazari2022MINE,RakotosonMossino,Sannipoli2022,TrombettiVazquez,Vazquez}. For the time being we refer to the monograph \cite{Kawohl} and to the survey of Talenti himself \cite{Talenti2016}. In general these comparison principles are expressed in terms of \emph{concentration of solutions}, using the order relation $\prec$ defined as follows: for a domain $\O$, for any non-negative functions $f\,, g \in L^1(\O)$, we say that 
\begin{multline}\label{Eq:Comp1}f\prec g\text{ if, and only if, for any $V\in (0;\mathrm{Vol}(\O)),$ }\sup_{E\subset \O\,, \mathrm{Vol}(E)=V}\int_E f\leq \sup_{E\subset \O\,, \mathrm{Vol}(E)=V}\int_E g.
\end{multline}
This relation can be expressed using the Schwarz rearrangement, see Definitions \ref{De:1}-\ref{De:2} below. The content of any Talenti-type inequality is that if we consider a parabolic or an elliptic equation of the form $\mathcal Lu=f$ then we can compare the solution $u$ with a solution $\tilde u$ of a related  equation $\tilde{\mathcal L}\tilde u=\tilde f$ in the ball, where the tilde $\tilde\cdot$ simply means that certain coefficients of the equation were symmetrised.

While most of the works we cited above deal with rearrangements in space (\emph{i.e.} for certain criteria is it better to have symmetric in space source terms/advection matrices?) it is interesting to investigate the influence of \emph{time-rearrangement} of functions: if we are working with a parabolic equation, is there a \emph{good} way to rearrange the source term both in time and space? In this paper, we prove that the answer to this question is \emph{no} and that rearranging source terms in time can not yield as strong concentration results as rearranging source terms in space. 
\subsection{Rearrangement and order relation}
To fix notations, let a dimension $d\in \N\backslash\{0\}$ and a radius $R>0$ be fixed, and consider the ball $\O:=\B(0;R)$ in $\R^d$. {We will use the notation $\mathscr C^\infty(\O)$ to denote the set of infinitely differentiable functions in $\O$.}
\begin{definition}\label{De:1}
For any non-negative function $g\in L^2(\O)$, there exists a unique radially symmetric, non-negative, non-increasing function $g^\#\in L^2(\O)$ that has the same distribution function as $g$ \emph{i.e.}
\[ \forall t\geq 0\,, \mathrm{Vol}\left(\{g\geq t\}\right)=\mathrm{Vol}\left(\{g^\#\geq t\}\right).\] $g^\#$ is called the Schwarz rearrangement of $g$.
\end{definition}
There are two famous inequalities that are related to the Schwarz rearrangement:
\begin{enumerate}
\item First, the P\'{o}lya-Szeg\"{o} inequality, which states that, if $f\in W^{1,2}(\O)$ { is a non-negative function }, then $f^\#\in W^{1,2}(\O)$ and, furthermore, that we have 
\begin{equation}\label{Eq:PS}\int_\O\left| \n f^\#\right|^2\leq \int_\O |\n f|^2.\end{equation}
\item Second, the Hardy-Littlewood inequality: it states that, if $f\,, g\in L^1(\O)$ are non-negative functions then 
\begin{equation}\label{Eq:HL}\int_\O fg\leq \int_\O f^\#g^\#.
\end{equation}
\end{enumerate}
The Schwarz rearrangement allows to reformulate the comparison relation \eqref{Eq:Comp1}:
\begin{definition}\label{De:2}
For any non-negative $f,g\in L^2(\O)$, we say that $g$ dominates $f$, and we write $f\prec g$ if, and only if
\[\forall r\in (0;R)\,,\int_{\B(0;r)}f^\#\leq \int_{\B(0;r)}g^\#.\]
\end{definition}
It is easily checked that this definition is equivalent to \eqref{Eq:Comp1} since one can check that, by equi-measurability of $f$ and of its Schwarz rearrangement, and since $f^\#$ is radially non-increasing, there holds
\[ \forall V \in (0;\mathrm{Vol}(\O))\,, \sup_{E\subset \O\,, \mathrm{Vol}(V)}\int_\O f=\int_{\B(0;r_V)}f^\#\text{ with }\mathrm{Vol}(\B(0;r_V))=V.\]

\subsection{Parabolic model, problem under scrutiny and main result}
The model under scrutiny in this paper is a linear heat equation: for any $f\in L^\infty(\OT)$, we let $u_f$ be the only solution of the linear heat equation 
\begin{equation}\label{Eq:Main}
\begin{cases}
\frac{\partial u_f}{\partial t}-\Delta u_f=f&\text{ in }\OT\,, 
\\ u_f(t,\cdot)=0&\text{ on }(0;T)\times\partial \O\,, 
\\u_f(0,\cdot)=0&\text{ in }\O.\end{cases}
\end{equation}The classical isoperimetric parabolic inequality \cite{RakotosonMossino,Vazquez} asserts the following: denoting, for a given $f=f(t,x)$ the spatially rearranged function $f^\#$ as 
\[f^\#:\OT\ni (t,x)\mapsto (f(t,\cdot))^\#(x)\] we have
\[\forall t\in [0;T]\,, u_f(t,\cdot)\prec u_{f^\#}(t,\cdot).\]
This quite naturally leads to the question: can such estimates be reached when we rearrange $f$ not only in space, but also in time? Note that this Talenti inequality implies the existence of a maximal element for the order relation $\prec$ in the following sense: let $\delta:[0;T]\to (0;\mathrm{Vol}(\O))$ be a function that models a time-dependent volume constraint and consider the set 
\[ \mathcal F_\delta:=\left\{f\in L^\infty(\OT): 0\leq f\leq 1\text{ a.e.,  and for a.e. } t\in [0;T]\,, \int_\O f(t,\cdot)=\delta_t\right\}.\]Let $\overline f_\delta$ be defined as
\[\overline f_\delta:(t,x)\mapsto \mathds 1_{\B(0;r_{\delta(t)})}(x)\text{ where $r_{\delta(t)}$ is chosen so that $\mathrm{Vol}(\B(0;r_{\delta(t)}))=\delta(t)$}.
\]
If we define
\[ \mathcal H_\delta(T):=\{u_f(T,\cdot)\,, f\in \mathcal F_\delta\}\subset L^2(\O)\]
then the parabolic Talenti inequality implies that, for any $T>0$, $u_{\overline f_\delta}$ is a $\prec$-maximal element in $\mathcal H_\delta$.

Our question here is the following: can we obtain maximal elements in a wider class of source terms where, unlike in the definition of $\mathcal F_\delta$, we do not impose, for every time, a volume constraint? Let us thus introduce, for a given volume constraint $V_0\in (0;\mathrm{Vol}(\OT))$, the class of admissible controls 

\begin{equation}\label{Eq:Adm}\tag{$\bold{Adm}$}
\mathcal F:=\left\{f\in L^\infty(\OT)\,, 0\leq f\leq 1\text{ a.e., }\iint_\OT f=V_0\right\}.\end{equation} 
Our question is then: defining, for any $T>0$, 
\[ \mathcal H(T):=\left\{u_f(T,\cdot)\,, f\in \mathcal F\right\},\] \emph{does there exist a $\prec$-maximal element in $\mathcal H(T)$}? In other words, does there exists a $f^*\in \mathcal F$ such that:

\begin{equation}\label{Eq:Ibey}
\forall f \in \mathcal F\,, u_f(T,\cdot)\prec u_{f^*}(T,\cdot)?\end{equation}
Here, the answer is no:
\begin{theorem}\label{Th:Main}
There exists no $f^*\in \mathcal F$ such that \eqref{Eq:Ibey} holds.\end{theorem}
%
%
%
%
%

\section{Proof of theorem \ref{Th:Main}}
\paragraph{Strategy of proof and auxiliary problems} To prove the result we will argue by contradiction and assume that there exists $f^*\in \mathcal F$ such that \eqref{Eq:Ibey} holds for a certain time horizon $T>0$. By the parabolic Talenti inequality, we may assume that $f^*=(f^*)^\#$ so that $u_{f^*}=u_{f^*}^\#$. By definition of $f^*$ we know that, for any $f\in \mathcal F$ and any $r\in [0,R]$, 
\[\int_0^r{\xi^{d-1}} u_f^\#(t,{\xi})d{\xi}\leq \int_0^r {\xi^{d-1}}u_{f^*}(t,{\xi})d{\xi}.\] In particular, for any $r\in [0;R]$, ${f^*}$ is a solution of the optimisation problem
\begin{equation}\label{Eq:Pvr}\tag{$P(r)$}\max_{f\in \mathcal F}\left(\max_{E\subset \O\,, \operatorname{Vol}(E)=\omega_d r^d}\int_E u_f\right),\end{equation} where $\omega_d=\mathrm{Vol}(\B(0,1)).$

To prove Theorem \ref{Th:Main}, it suffices to show that no $f^*\in \mathcal F$ can solve \eqref{Eq:Pvr} for all $r\in[0;R]$.

\paragraph{Proof of Theorem \ref{Th:Main}}
Following the discussion above we prove the following result:
\begin{lemma}\label{Le:Varphi}
Let $f^*\in \mathcal F$ be such that, for any $r\in (0,R)$, $f^*$ is a solution of \eqref{Eq:Pvr}. Then, for any radially symmetric, {non-increasing}, non-negative function $\varphi\in \mathscr C^\infty(\O)$, $f^*$ is a solution of 
\[ \max_{f\in \mathcal F}\int_\O u_f(T,\cdot)\varphi.\]
\end{lemma}
\begin{proof}[Proof of lemma \ref{Le:Varphi}]
Let us fix $\varphi$  in the conditions of the lemma. We can approximate $\varphi$ by an increasing sequence of radially symmetric step-functions {$\{\phi_k\}_{k\in \N}$ as follows:
define, for an integer $k\geq 1$,
\[r_{k,j}:=\frac{j}{k}R \quad (j=0,\dots,k)\,, \alpha_{k,j}:=\varphi(r_{k,j+1}) \quad (j=0,\dots,k-1)
\]
and set 
\[ \phi_k:=\sum_{j=0}^{k-1}\alpha_{k,j}\mathds 1_{\B(0;r_{k,j+1})\backslash \B(0;r_{k,j})}.\]
However, from this decomposition it appears that we may rewrite $\phi_k$ as 
\[ \phi_k=\sum_{j=1}^k\beta_{k,j}\mathds 1_{\B(0;r_{k,j})}\text{ where, for any $j\in \{0,\dots,k\}$, $\beta_{k,j}\geq 0$.}\]
Indeed it suffices to define the coefficients $\beta_{k,j}$ as 
\[ \beta_{k,k}:=\alpha_{k,k-1}\text{ and, for any $j\in \{1,\dots,k-1\}$, } \beta_{k,j}:=\alpha_{k,j-1}-\alpha_{k,j}\geq 0\]where the last inequality comes from the fact that $\varphi$ is non-increasing.
Consequently, for any $k\in \N$ and any $j\leq k$, \[ \beta_{k,j}\int_{\B(0,r_{k,j})} u_f(T,\cdot)\leq \beta_{k,j}\int_{\B(0,r_{k,j})} u_{f^*}(T,\cdot)\] by the definition of $f^*$.} Passing to the limit $k\to \infty$ yields the result.

\end{proof}
We single out the following optimisation problem defined for any $\varphi\in \mathscr C^\infty(\O)$:
\begin{equation}\label{Eq:PvVarphi}\tag{$P_\varphi$}
\max_{f\in \mathcal F}\int_\O u_f(T,\cdot)\varphi.\end{equation}

{To prove Theorem \ref{Th:Main}, we will need to characterise the optimisers of \eqref{Eq:PvVarphi} in certain cases. Such a characterisation can be obtained by studying the optimality conditions for \eqref{Eq:PvVarphi} which is what we now set out to do.}

{\textbf{Optimality conditions for \eqref{Eq:PvVarphi}:}}

{Define $p_\varphi$ as the unique solution of the backward heat equation 
\begin{equation}\label{Eq:Adjoint}
\begin{cases}
{\frac{\partial p_\varphi}{\partial t}}+\Delta p_\varphi=0&\text{ in }\OT\,, 
\\p_\varphi(t,\cdot)=0&\text{ on }[0;T]\times \partial \O\,, 
\\ p_\varphi(T,\cdot)=\varphi&\text{ in }\O\,,\end{cases}
\end{equation}
Multiplying \eqref{Eq:Main} by $p_\varphi$ and integrating by parts we obtain 
\begin{equation}
\forall f \in \mathcal F\,, \int_\O u_f(T,\cdot)\varphi=\iint_\OT f p_\varphi.
\end{equation}
The function $p_\varphi$ encodes the optimality conditions for \eqref{Eq:PvVarphi}. To further characterise optimisers we need some information on the level sets of the function $p_\p$. Such information is given in the following lemma:
\begin{lemma}\label{Le:RegPphi}
Assume  $\p\in \mathscr C^\infty(\O)\cap W^{1,2}_0(\O)$, $\p=\p^\#$, {$\varphi\geq 0$} and $\p$ is not constant. Then, for any $t\in [0;T)$ and for any $\tau\in \left(0;\Vert \p\Vert_{L^\infty(\O)}\right)$ the level set $\{p_\p(t,\cdot)=\tau\}$ is {a} $(d-1)$-dimensional sphere.
\end{lemma}
\begin{proof}[Proof of Lemma \ref{Le:RegPphi}]
Since $\p\in \mathscr C^\infty(\O)$, standard parabolic estimates imply that $p_\phi \in \mathscr C^\infty(\OT)$. Since $\p$ is radially symmetric, so is $p_\p$. By the maximum principle, for any $t\in [0;T]$,  
\[ \frac{\partial p_\p}{\partial \nu}(t,\cdot) \leq 0\text{ on }\partial \O.\] Let $q_\p:=\frac{\partial p_\p}{\partial r}.$ We already know that 
\[\forall t \in [0;T]\,, q_\p(t,\cdot) \leq 0\text{ on }\partial \O.\] Furthermore, at $t=T$, since $\p$ is not constant and radially symmetric, non-increasing, 
\[ q_\p(T,\cdot)\leq 0\,, q_\p(T,\cdot)\neq 0.\] Differentiating \eqref{Eq:Adjoint} with respect to $r$ we get the following equation 
\[\begin{cases}
\frac{\partial q_\p}{\partial t}+\Delta q_\p=0&\text{ in }\OT\,, 
\\ q_\p\leq 0&\text{ on } [0;T]\times \partial \O\,, 
\\ q_\p(T,\cdot)\leq 0\,, q_\p(T,\cdot)\neq 0&\text{ in }\O.\end{cases}\]
By the strong maximum principle it follows that 
\[ \forall t<T\,, q_\p(t,\cdot)<0 \text{ in }\O.\] Thus, for any $t\in [0;T)$, $p_\p(t,\cdot)$ is radially decreasing. In particular, its level sets have zero Lebesgue measure and coincide with spheres.
\end{proof}
Now let us turn back to the optimality conditions for \eqref{Eq:PvVarphi}: let $f_\varphi$ be a solution of \eqref{Eq:PvVarphi}. From the bathtub principle \cite[Theorem 1.14]{LiebLoss} and the fact that $p_\varphi$ only has level sets of measure zero, it follows that there exists a Lagrange multiplier $c_\varphi\in \R$ such that, up to negligible sets,
\begin{enumerate}
\item $\{(t,x)\in \OT:\, f_\varphi(t,x)=1\}= \{(t,x)\in \OT:\, p_\varphi(t,x)>c_\varphi\}$,
\item $\{(t,x)\in \OT:\, f_\varphi(t,x)=0\}=\{(t,x)\in \OT:\, p_\varphi(t,x)<c_\varphi\}$,
\item $\{(t,x)\in \OT:\, 0<f_\varphi(t,x)<1\}$ has Lebesgue measure zero. \end{enumerate}
The constant $c_\p$ appearing is dubbed the \emph{Lagrange multiplier associated with $\p$}. We emphasise that it is a constant that depends neither on space nor on time. These conditions define $f_\varphi$ univocally. 
Furthermore, as the (time-space) dependent level-set satisfies $\mathrm{Vol}(\{f_\varphi=1\})\in (0;V_0)$ the maximum principle implies 
\begin{equation}\label{Eq:Sunrise}
0<c_\varphi<\Vert p_\varphi\Vert_{L^\infty(\OT)}\leq \Vert \varphi\Vert_{L^\infty(\O)}.\end{equation}
}

The following lemma essentially contains the proof of Theorem \ref{Th:Main}:
\begin{lemma}\label{Le:Cruise}
There exist two radially symmetric, decreasing {and non-negative functions} $\varphi\,, \psi\in \mathscr C^\infty(\O)$ such that $(P_{\varphi})$ and $(P_{\psi})$ do not have the same solutions.
\end{lemma}
\begin{proof}[Proof of Lemma \ref{Le:Cruise}]

\textbf{Construction of $\varphi\,, \psi$ such that $f_\p\neq f_\psi$}
Let $\varphi$ be a {cut-off} function; in other words, $\varphi$ satisfies:
\begin{itemize}
\item$\varphi\in \mathscr C^\infty(\O,\R_+)$ is a radially symmetric, non-increasing function.
\item $\varphi \equiv 1$ on $\B(0,R/8)$ and is radially decreasing on $\mathbb B(0,R/4)\backslash \mathbb B(0,R/8)$.
\item$\varphi\equiv 0$ on $\B(0,R)\backslash \B(0,R/4)$.
\end{itemize}Similarly we pick $\psi $  that satisfies 
\begin{itemize}
\item$\psi\in \mathscr C^\infty(\O,\R_+)$ is a radially symmetric, non-increasing function.
\item $\psi \equiv 1$ on $\B(0,R/2)$ and is radially decreasing on $\mathbb B(0,R/2)\backslash \mathbb B(0,3R/4)$.
\item$\psi\equiv 0$ on $\B(0,R)\backslash \B(0,3R/4)$.\end{itemize}

We claim that for this $\psi$ and this $\p$ we have 
 $f_\varphi\neq f_\psi.$
 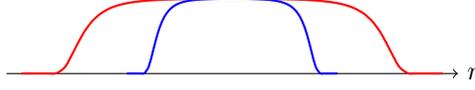
\begin{figure}
  \begin{center}

 \begin{tikzpicture}
  \draw[->] (-3, 0) -- (3, 0) node[right] {$r$};
  \draw[scale=0.5, domain=0.2:3, smooth, variable=\x, blue,thick] plot ({\x-3}, {2^(1-1/(\x*\x*\x*\x*\x)});
    \draw[scale=0.5, domain=0.2:3, smooth, variable=\x, blue,thick] plot ({-\x+3}, {2^(1-1/(\x*\x*\x*\x*\x))});
    \draw[scale=0.5, domain=0.2:3, smooth, variable=\x, red,thick] plot ({(\x-3)*2}, {2^(1-1/(\x*\x*\x*\x*\x)});
    \draw[scale=0.5, domain=0.2:3, smooth, variable=\x, red,thick] plot ({(-\x+3)*2}, {2^(1-1/(\x*\x*\x*\x*\x)});
\end{tikzpicture}
\caption{In blue, the graph of $\p$. In red, the graph of $\psi$.}
\end{center}

\end{figure}

To see why $f_\varphi\neq f_\psi$, let $c_\varphi\,, c_\psi$ be the Lagrange multipliers associated, respectively, with $\varphi$ and $\psi$. {Recall that \eqref{Eq:Sunrise} gives
\[ 0<c_\p\,, c_\psi<\max\left(\Vert \varphi\Vert_{L^\infty(\O)},\Vert \psi\Vert_{L^\infty(\O)}\right)=1.\] Define, for any $r>0$, $\mathbb S(0;r)$ as the $(d-1)$-dimensional sphere or radius $r$ (\emph{i.e.} $\mathbb S(0;r)=\{x\in \R^d,\, \Vert x\Vert =r\}$) and let, } for any $t\in (0;T)$, $r_\varphi(t)$ (resp. $r_\psi(t)$) be such that 
\[\{p_\varphi(t,\cdot)=c_\varphi\}=\mathbb S(0,r_\varphi(t)) \text{ (resp. }\{p_\psi(t,\cdot)=c_\psi\}=\mathbb S(0,r_\psi(t))  \text{)}.\] As $c_\p\,, c_\psi>0$ we have 
\[ \sup_{t\in [0;T]}r_\p(t)\,, r_\psi(t)<{R}.\]
As  we also have
\[ c_\varphi\,, c_\psi<1\]  we deduce
\[\inf_{t\in[0;T]}r_\varphi(0)\,, r_\psi(0)>0.\] Since, by parabolic regularity, $p_\varphi$ and $p_\psi$ are $\mathscr C^\infty$ and radially decreasing in the sense that $\partial_r p_\psi\,, \partial_r p_\varphi<0$ in $\OT$, $r_\varphi$ and $r_\psi$ are continuous\footnote{One could also observe that $r_\p$ solves the differential equation $dr_\p/dt=-\partial_tp/\partial_r\p$. Since $r_\p$ is uniformly bounded away from 0, we also get the fact that $r_\p$ is $\mathscr C^1$.} in $[0;T]$. 

Finally, since $\varphi=0$ on $\mathbb B(0,R)\backslash \B(0,R/4)$ we have $r_\varphi(T)<R/4$. Similarly, since $\psi\equiv 1$ on $\mathbb B(0,R/2)$ we have $r_\psi(T)> R/2$. Consequently, ${r_\varphi}\neq r_\psi$ in a neighbourhood of $T$. {But now recall that from the optimality conditions of ($P_\varphi$)-($P_\psi$), we have
\[ f_\varphi(t,x)=\mathds 1_{\B(0;r_\varphi(t))}(x)\,, f_\psi(t,x)=\mathds 1_{\B(0;r_\psi(t))}(x).\] As $r_\varphi\neq r_\psi$ in a neighbourhood of $T$, } $f_\varphi\neq f_\psi$. {This concludes the proof of the Theorem.}

\end{proof}
%

\bibliographystyle{abbrv}
\bibliography{BiblioTwoPhase-Parabolic}

\begin{thebibliography}{10}

\bibitem{Alvino1986}
A.~Alvino, P.~Lions, and G.~Trombetti.
\newblock A remark on comparison results via symmetrization.
\newblock {\em Proceedings of the Royal Society of Edinburgh: Section A
  Mathematics}, 102(1-2):37--48, 1986.

\bibitem{Alvino1990}
A.~Alvino, P.-L. Lions, and G.~Trombetti.
\newblock Comparison results for elliptic and parabolic equations via {S}chwarz
  symmetrization.
\newblock {\em Annales de l'Institut Henri Poincare (C) Non Linear Analysis},
  7(2):37--65, Mar. 1990.

\bibitem{alvino1991}
A.~Alvino, P.-L. Lions, and G.~Trombetti.
\newblock Comparison results for elliptic and parabolic equations via
  symmetrization: a new approach.
\newblock {\em Differential Integral Equations}, 4(1):25--50, 1991.

\bibitem{AlvinoTrombettiLions}
A.~Alvino, P.-L. Lions, and G.~Trombetti.
\newblock Comparison results for elliptic and parabolic equations via
  symmetrization: a new approach.
\newblock {\em Differential Integral Equations}, 4(1):25--50, 1991.

\bibitem{AlvinoNitschTrombetti}
A.~Alvino, C.~Nitsch, and C.~Trombetti.
\newblock A {T}alenti comparison result for solutions to elliptic problems with
  {R}obin boundary conditions, 2019.

\bibitem{Alvino2010}
A.~Alvino, R.~Volpicelli, and B.~Volzone.
\newblock Comparison results for solutions of nonlinear parabolic equations.
\newblock {\em Complex Variables and Elliptic Equations}, 55(5-6):431--443,
  Apr. 2010.

\bibitem{Bandle}
C.~Bandle.
\newblock {\em Isoperimetric Inequalities and Applications}.
\newblock Monographs and studies in mathematics. Pitman, 1980.

\bibitem{Chiacchio}
F.~Chiacchio.
\newblock {Comparison results for linear parabolic equations in unbounded
  domains via Gaussian symmetrization}.
\newblock {\em Differential and Integral Equations}, 17(3-4):241 -- 258, 2004.

\bibitem{Hamel2011}
F.~Hamel, N.~Nadirashvili, and E.~Russ.
\newblock Rearrangement inequalities and applications to isoperimetric problems
  for eigenvalues.
\newblock {\em Annals of Mathematics}, 174(2):647--755, sep 2011.

\bibitem{Kawohl}
B.~Kawohl.
\newblock {\em Rearrangements and Convexity of Level Sets in {PDE}}.
\newblock Springer Berlin Heidelberg, 1985.

\bibitem{Kesavan1988}
S.~Kesavan.
\newblock Some remarks on a result of talenti.
\newblock {\em Annali della Scuola Normale Superiore di Pisa - Classe di
  Scienze}, 15(3):453--465, 1988.

\bibitem{Kesavan}
S.~Kesavan.
\newblock {\em Symmetrization and Applications}.
\newblock {WORLD} {SCIENTIFIC}, Apr. 2006.

\bibitem{Langford}
J.~Langford.
\newblock {\em Comparison Theorems in Elliptic Partial Differential Equations
  with Neumann Boundary Conditions}.
\newblock PhD thesis, Washington University, 2012.

\bibitem{LiebLoss}
E.~Lieb and M.~Loss.
\newblock {\em Analysis}.
\newblock American Mathematical Society, Providence, Rhode Island, 2001.

\bibitem{Mazari2022}
I.~Mazari.
\newblock Quantitative estimates for parabolic optimal control problems under
  l$\infty$ and l1 constraints in the ball: Quantifying parabolic isoperimetric
  inequalities.
\newblock {\em Nonlinear Analysis}, 215:112649, Feb. 2022.

\bibitem{Mazari2022MINE}
I.~Mazari.
\newblock Some comparison results and a partial bang-bang property for
  two-phases problems in balls.
\newblock {\em Mathematics in Engineering}, Special Issue: Calculus of
  Variations and Nonlinear Analysis (Ed: D. Mazzoleni and B. Pellacci), 2022.

\bibitem{RakotosonMossino}
J.~Mossino and J.~M. Rakotoson.
\newblock Isoperimetric inequalities in parabolic equations.
\newblock {\em Annali della Scuola Normale Superiore di Pisa - Classe di
  Scienze}, Ser. 4, 13(1):51--73, 1986.

\bibitem{Sannipoli2022}
R.~Sannipoli.
\newblock Comparison results for solutions to the anisotropic laplacian with
  robin boundary conditions.
\newblock {\em Nonlinear Analysis}, 214:112615, Jan. 2022.

\bibitem{Talenti}
G.~Talenti.
\newblock Elliptic equations and rearrangements.
\newblock {\em Annali della Scuola Normale Superiore di Pisa - Classe di
  Scienze}, Ser. 4, 3(4):697--718, 1976.

\bibitem{Talenti2016}
G.~Talenti.
\newblock The art of rearranging.
\newblock {\em Milan Journal of Mathematics}, 84(1):105--157, Mar. 2016.

\bibitem{TrombettiVazquez}
G.~Trombetti and J.~L. Vazquez.
\newblock A symmetrization result for elliptic equations with lower-order
  terms.
\newblock {\em Annales de la Facult\'e des sciences de Toulouse :
  Math\'ematiques}, Ser. 5, 7(2):137--150, 1985.

\bibitem{Vazquez}
J.~L. Vazquez.
\newblock Sym{\'e}trisation pour $u_t={\Delta}\varphi(u)$ et applications.
\newblock {\em C. R. Acad. Sci. Paris S{\'e}r. I Math.}, 295, 1982.

\end{thebibliography}

\end{document}